\documentclass{amsproc}

\usepackage{bbm}

\newtheorem{theorem}{Theorem}[section]
\newtheorem{lemma}[theorem]{Lemma}

\theoremstyle{definition}
\newtheorem{definition}[theorem]{Definition}

\theoremstyle{remark}
\newtheorem{remark}[theorem]{Remark}

\numberwithin{equation}{section}

\newcommand\B{\mathbb{B}}
\newcommand\C{\mathbb{C}}
\newcommand\UH{\mathbb{H}}
\newcommand\N{\mathbb{N}}
\newcommand\Z{\mathbb{Z}}
\newcommand\R{\mathbb{R}}
\newcommand\cs{\textup{cs}}
\newcommand\cp{\textup{cap}}
\newcommand\diam{\textup{diam}}
\newcommand\Lip{\textup{Lip}}
\newcommand\lip{\textup{lip}}
\newcommand\VMO{\textup{VMO}}

\newcommand\BMO{\textup{BMO}}
\newcommand\side{\textup{side}}
\newcommand\spn{\textup{span}}
\newcommand\loc{\textup{loc}}
\newcommand\dist{\textup{dist}}
\newcommand\spt{\textup{supp}}
\newcommand\Vit{\mathfrak{T}}
\newcommand\Cau{\mathfrak{C}}
\newcommand\Test{\mathcal{D}}
\newcommand\dd[2]{\displaystyle\frac{\partial#1}{\partial#2}}
\newcommand\dds[2]{\frac{\partial#1}{\partial#2}}
\newcommand\ddx{\dd{}{x}}
\newcommand\ddy{\dd{}{y}}
\newcommand\ddz{\dd{}{z}}
\newcommand\ddbz{\dd{}{\bar{z}}}
\newcommand\ddbzs{\dds{}{\bar{z}}}
\newcommand\half{{\scriptstyle\frac12}}
\newcommand\dsty{\displaystyle}

\RequirePackage{bbm}
\newcommand\ONE{\mathbbm{1}}

\def\hookto{{\ \hookrightarrow\ }}
\def\locin{\ {\buildrel{\rm loc}\over\hookrightarrow}\ }
\def\loceq{\ {\buildrel{\rm loc}\over=}\ }

\newcommand\sint{{\textstyle\int}}

\newcommand\beq{\begin{equation}}
\newcommand\eeq{\end{equation}}

\begin{document}

\title[Boundary values]{Boundary values of holomorphic
distributions in negative Lipschitz classes}


\author{Anthony G. O'Farrell}
\address{Mathematics and Statistics\\NUI, Maynooth\\Co Kildare\\W23 HW31\\Ireland}
\curraddr{}
\email{anthony.ofarrell@mu.ie}
\thanks{}



\date{}

\begin{abstract}
We consider the behaviour at a boundary point of an
open subset $U\subset\C$ of distributions that
are holomorphic on $U$ and belong to what are called negative
Lipschitz classes.  The result explains the significance
for holomorphic functions of series of Wiener type involving Hausdorff
contents of dimension between $0$ and $1$. We begin with a survey about function spaces and capacities
that sets the problem in context and reviews the relevant
general theory.
\end{abstract}

\maketitle


\section{Introduction}
\subsection{Boundary values}
It may happen that all bounded holomorphic functions on an
open set $U\subset\C$ admit a \lq reasonable boundary value'
at some boundary point.  This was first noted by
Gamelin and Garnett \cite{GG}.  The condition for the
existence of such a boundary value is expressed
using a series of \lq Wiener type', and involves the Ahlfors
analytic capacity, $\gamma$. The condition is
$$ \sum_{n=1}^{\infty} 2^n \gamma(A_n\setminus U) < +\infty.$$
Here, if $b$ is the boundary point in question,
$A_n$ denotes the annulus
$$ A_n(b):= \left\{ z\in\C: 
\frac1{2^{n+1}} \le |z-b| \le \frac1{2^n} \right\}. 
$$
This condition says that in an appropriate
sense the complement of $U$ is very thin at $b$;
in particular it implies that $U$ has full area
density at $b$, i.e.
$$ \lim_{r\downarrow0}\frac{|\B(b,r)\cap U)|}{\pi r^2} =1, $$
where we denote the area of a set $E\subset\C$ by $|E|$.
When the series converges, 
it is emphatically 
\emph{not} the case that
the limit
$$ \lim_{x\to a, z\in U} f(z) $$
exists for all functions $f$ bounded and holomorphic on 
$U$ (unless all such functions extend holomorphically
to a neighbourhood of $b$).  But for each such function there is a 
(unique) value which we may call  $f(a)$, with the 
property that for some set $E\subset U$ having
full area density at the point $a$ 
$$ \lim_{x\to  a, z\in E} f(z) = f(a). $$

\subsection{Peak points}
This result is one of many about the boundary
behaviour of analytic and harmonic functions
on arbitrary open sets.  The original Wiener
series (cf. \cite{Wiener} or \cite{Gardiner}) 
involved logarithmic capacity in dimension
two, Newtonian capacity in dimension three,
and other Riesz capacities in higher dimensions,
and characterised boundary points that are
regular for the Dirichlet problem.  Later
these points were recognised as \emph{peak points}
for the space of functions
harmonic on an open set $U$ and continuous on
its closure.  The first person to use such a series
with holomorphic functions was Melnikov 
\cite[Theorem VIII.4.5]{Gamelin}. He characterised the peak points
for the uniform closure on a compact set $X\subset\C$
of the algebra of all rational functions having
poles off $X$. He used the Ahlfors capacity,
and he showed that a point $b\in X$ is a peak
point if and only if
$$ \sum_{n=1}^{\infty} 2^n \gamma(\mathring{A_n}\setminus X) < +\infty.$$
(This was used by Gamelin and Garnett
to obtain their above-quoted result.)

For a bounded open set $U\subset\C$, and a point $b\in\partial U$,
the condition 
$$ \sum_{n=1}^{\infty} 2^n \alpha({A_n}\setminus U) < +\infty,$$
where $\alpha$ denotes the so-called \emph{continuous}
analytic capacity (introduced by Dolzhenko) is necessary and
sufficient for $b$ to be a peak point for the
algebra of all continuous functions on $\overline{U}$, 
holomorphic on $U$ \cite{Gamelin}.

\subsection{Capacities}
The vague idea that 
\emph{there is a capacity for every problem}
has gathered momentum over time.  A \emph{capacity}
is a function $c$ that assigns nonnegative extended real numbers
to sets, and is nondecreasing: 
$$ E_1 \subset E_2 \implies c(E_1)\le c(E_2).$$
Keldysh \cite{Keldysh} used
Newtonian capacity to solve the problem
of stability for the Dirichlet problem.
Vitushkin used analytic capacity to solve the problem
of uniform rational approximation \cite[Chapter VIII]{Gamelin}. 
Vitushkin's theorem
is completely analogous to Keldysh's: harmonic functions
have been replaced by holomorphic functions,
and Newtonian capacity by analytic capacity.
The same switch relates Wiener's regularity criterion
and Melnikov's peak-point criterion.

In an influential little book \cite{Carleson},
Carleson explained how other capacities
(particularly kernel capacities)
could be used to solve
problems about boundary values, convergence of
Fourier series, and removable singularities,
and in an appendix (prepared by Wallin) 
he listed over a thousand articles from
Mathematical Reviews up to 1965 that involve
some combination of these ideas.

\subsection{Continuous point evaluations}In relation to $L^p$ holomorphic approximation,
the appropriate capacity is a condenser
capacity.  The groundwork on condenser capacities
and (generalized) extremal length had already
been laid down by Fuglede \cite{Fuglede}.
Hedberg \cite{Hedberg-approx} (see also \cite{Bagby})
worked out the analogue of Vitushkin's
theorem for $L^p$ approximation on compact
$X\subset\C$, and \cite{Hedberg-bpd}
proved the analogue of
Melnikov's theorem.  Hedberg's result is about
\emph{continuous point evaluations}.  To explain this
concept, consider a Banach space $F$ of \lq functions'
on some set $E\subset\C$, where each element
$f\in F$ is defined almost-everywhere on $E$
with respect to area measure $m$. Suppose $b\in\overline{E}$
and the subspace $F_b$, consisting of those $f\in F$
that extend holomorphically to some neighbourhood
of $b$, is a dense subset of $F$. Then we say that
$F$ \emph{admits a continuous point evaluation at $b$}
if there exists $\kappa>0$ such that
$$    |f(b)| \le \kappa \|f\|_F, \ \forall f\in F_b. $$
This means that the functional $f\mapsto f(b)$
has a continuous extension from $F_b$ to the whole
of $F$.  Taking the case where $F$ is the closure
$R^p(X)$ in $L^p(X,m)$ of the rational functions with poles
off a compact $X\subset\C$, Hedberg showed that
if $2<p<+\infty$, then $R^p(X)$ admits a continuous point evaluation at
$b$ if and only if
$$ \sum_{n=1}^{\infty} 2^{nq} 
\Gamma_q({A_n(b)}\setminus X) < +\infty.$$
Here $q=p/(p-1)$ is the conjugate index, and $\Gamma_q$
is a certain condenser capacity.  When $p<2$,
$R^p(X)$ never admits a continuous point evaluation
at $b$, unless $b$ is an interior point of $X$.
In the case $p=2$, Hedberg left an interesting
gap between the sharpest known necessary condition
and the sharpest known sufficient condition, 
and this gap was closed by Fernstr\"om \cite{Fernstrom1975}.   
Historically,
the existence of continuous point evaluations in 
the $L^2$ case attracted considerable attention, because
of hopes that it might provide a way to
attack the invariant subspace problem for
operators on Hilbert space, and hopes that it might
provide a way to attack
the $L^2$ rational approximation problem
\cite{Brennan1, Brennan2, Brennan3}.  

The existence problem for continuous point evaluations
at boundary points has also been studied for harmonic functions
in the Sobolev space $W^{1,2}$, and Kolsrud
\cite{Kolsrud} gave a solution in terms of Wiener
series.  

In the literature, 
continuous point evaluations are often referred
to as \emph{bounded} point evaluations.

\subsection{Continuous point derivations}
There are similar results about the possibility
that the $k$-th complex derivative $f\mapsto f^{(k)}$  
may have a continuous extension from $F_b$ to all
of $F$.  These involve the same Wiener series
as continuous point evaluations, except that
the base $2$ is replaced by $2^{k+1}$.  For instance,
the $R^p(X)$ result (also due to Hedberg) involves
the condition
$$ \sum_{n=1}^{\infty} 2^{(k+1)qn} 
\Gamma_q({A_n(b)}\setminus X) < +\infty.$$
The earliest such result was for
the uniform closure of the rationals,
and was due to Hallstrom \cite{Hallstrom}.

\subsection{Intrinsic capacities}
The present author began to formalize the pairing
of problems and capacities in his thesis \cite{OF-thesis}.
He considered the limited context of uniform
algebras on compact subsets of the plane.
To each functor $X\mapsto F(X)$ that associates
a uniform algebra to each compact $X\subset\C$,
and subject to certain coherence assumptions,
he associated a capacity 
$$\alpha(F,\cdot):\mathcal{O}\to
[0,+\infty),$$
where $\mathcal{O}$ is the topology of $\C$.
He then proved a \emph{Capacity Uniqueness
Theorem}, which stated that the map $F\mapsto \alpha(F,\cdot)$
is injective on the set of such functors,
i.e. the capacity determines the functor.
The \emph{Local Capacity Uniqueness Theorem}
states that two functors $F$ and $G$ have
$F(X)=G(X)$ for a given compact set $X$ if and only
if the capacities $\alpha(F,\cdot)$ and
$\alpha(G,\cdot)$ agree on all open subsets
of the complement of $X$.  Thus $F(X)\stackrel{?}=G(X)$
is a problem for which there are \emph{two}
capacities, not one!
Vitushkin's theorem on rational approximation
is the case when $F(X)$ is the uniform closure
of the rational functions having poles off $X$
and $G(X)$ is the algebra of all functions
continuous on $X$ and holomorphic on $\mathring{X}$.
This part of the thesis is unpublished,
mainly because the main result is essentially
equivalent to a theorem of Davie \cite{Davie},
obtained independently.  In another unpublished
chapter, the author established that the
results of Melnikov and Hallstrom extended to
all these $F(X)$, replacing the analytic capacity by
$\alpha(F,\cdot)$.  

Other work by Wang \cite{Wang} and the
author \cite{OF-spikes, OF-wrap}
established a link between equicontinuous
pointwise H\"older conditions at a boundary
point and series in which $2$
is replaced by $2^\lambda$ for a nonintegral
$\lambda>1$. For instance, H\"older conditions
of order $\alpha$ are related to the convergence
of series
such as 
$$\sum_{n=1}^\infty 2^{(1+\alpha)n}\gamma(A_n\setminus X).$$ 

Moving on from the uniform norm, the
author considered parallel questions for 
Lipschitz or H\"older norms.  Building on a result of
Dolzhenko, he established that the 
equivalent of continuous analytic capacity
is the \emph{lower $\beta$-dimensional
Hausdorff content} $M^\beta_*$, with
$\beta=\alpha+1$. (For $\beta>0$, the 
\emph{$\beta$-dimensional
Hausdorff content} $M^\beta(E)$
of a set $E\subset\R^d$ is defined to be the infimum
of the sums $\sum_{n=0}^\infty r_n^\beta$, 
taken over all countable coverings of $E$
by closed balls $(\B(a_n,r_n))_n$.  If we replace
$r_n^\beta$ by $h(r_n)$ for an increasing
function $h:[0,+\infty)\to[0,+\infty)$
we get the Hausdorff $h$-content $M_h(E)$.
The
\emph{lower $\beta$-dimensional
Hausdorff content} $M^\beta_*(E)$ 
is defined to be the supremum of $M_h(E)$,
taken over all $h$ such that $0\le h(r)\le r^\beta$
for all $r>0$, and $r^{-\beta}h(r)\to0$ as
$r\downarrow0$.)
He proved \cite{OF-Vitushkin}
the analogue of Vitushkin's theorem
for rational approximation. Later, Lord and he
\cite{OF-Lord} proved the analogue of
Hallstrom's theorem for boundary derivatives.
For the $k$-th derivative, this involved the 
series condition
$$ \sum_{n=1}^{\infty} 2^{(k+1)n} 
M^{\alpha+1}_*({A_n(b)}\setminus X) < +\infty.$$

\subsection{SCS}\label{SS:1.7}
Moving to a more general context, the author
introduced the notion of a \emph{Symmetric 
Concrete Space} $F$ on $\R^d$, and considered the relation
between problems about a given such space $F$,
in combination with an elliptic operator $L$,
and an appropriate associated capacity,
the $L$-$F$-$\cp$.  
A Symmetric Concrete Space (SCS) on $\R^d$ is a
complete locally-convex topological vector space $F$
over the field $\C$, such that
\begin{itemize}
\item $\mathcal{D}\hookto F \hookto \mathcal{D}^*$;
\item $F$ is a topological $\mathcal{D}$-module
under the usual product $\phi\cdot f$ of a test function and a 
distribution;
\item $F$ is closed under complex conjugation;
\item The affine group of $\R^d$ acts by composition
on $F$, and each compact set of affine maps
gives an equicontinuous family of composition
operators. 
\end{itemize}

Here $\mathcal{D}=C^\infty_\cs(\R^d,\C)$, is the space of
\emph{test functions} and $\mathcal{D}^*$ is its dual,
the space of distributions, $A\hookto B$
stands for \lq\lq \emph{$A\subset B$ and the inclusion map
is continuous"}.  (In fact, it is elementary that
if $A$ and $B$ are metrizable SCS, then $A\subset B$ implies
$A\hookto B$.)

A SCS is a Symmetric Concrete Banach Space (SCBS)
when it is normable and is equipped with a norm.

We shall be concerned only with the case $d=2$, and we
identify $\R^2$ with $\C$.

The various analytic capacities
are $\frac{\partial}{\partial\bar z}$-$F$-$\cp$ for particular $F$.
The author planned a book about this subject, but this
project has never been completed.  Some
extracts with useful ideas and results were published.
The most useful ideas concern localness. For SCS $F$ and $G$,
we define
$$ F_\loc:= 
\{f\in\mathcal{D}^*: \phi\cdot f\in F,\ \forall \phi\in
\mathcal{D}\}, $$
$$ F_\cs:= \{
\phi\cdot f:  f\in F \textup{ and } \phi\in
\mathcal{D}\}, $$
$$ F \locin G \iff F_\loc \hookto G_\loc, $$
$$ F \loceq G \iff F_\loc = G_\loc, $$
and observe that
$$ F \loceq F_\loc \loceq F_\cs. $$

Published results include the following:
\\(1)  A \emph{Fundamental Theorem
of Calculus for SCS} that are weakly-locally
invariant under Calderon-Zygmund operators 
\cite[Lemma 12]{OF-1-reduction}.
This says that 
$$ D\sint F \loceq \sint DF \loceq F, $$
where 
$$ DF:= \mathcal{D}+\spn\left\{\dd{f}{x_j}:1\le j\le d, f\in F \right\} $$
and 
$$ \sint F:= \left\{f\in\mathcal{D}^*: \dd{f}{x_j}\in F,\textup{ for }
1\le j\le d \right\}.$$
\\
(2)
A \emph{1-reduction principle} that allows
us to establish equivalences between
problems for different operators $L$
\cite[Theorem 1]{OF-1-reduction}.
The identity operator $\ONE:f\mapsto f$ is elliptic.
If $U$ is open, then the equation $\ONE f=0$ on $U$
just means that $U\cap\spt(f)=\emptyset$.
The idea is to reduce questions about
$L$ and some space $F$ to equivalent problems about $\ONE$
and the space $LF:=\{Lf:f\in F\}$. 
\\
(3)
A general Sobolev-type embedding theorem \cite{OF-order}
involving the concept of the \emph{order}
of an SCS, and 
\\
(4) A theorem that says that
in dimension two all SCS are essentially
(technically, weakly-locally-) T-invariant
\cite{OF-T-invariance}, i.e. invariant
under the Vitushkin localization operators
(see Section \ref{S:proofs} below). 

In 1990
the author circulated a set of notes on the
concept of SCS and the main examples. 
Some ideas from these papers
were expounded by
Tarkhanov in his book on the 
\textit{Cauchy Problem
for Solutions of Elliptic Equations} \cite[Chapter 1]{Tarkhanov}.

The general point of view raised many
particular questions, and some of these have been
solved, while other loose ends remain.

\section{The Problem}
Our objective in the present paper is to address
a loose end connected to the results
on boundary behaviour of holomorphic functions
mentioned above.  For $0<\alpha<1$, the 
$\ddbzs$-$F$-$\cp$
associated to the Lipschitz class
$\Lip\alpha$
is $M^{\alpha+1}$, and that associated to the 
little Lipschitz class $\lip\alpha$
is $M_*^{\alpha+1}$.  Kaufmann \cite{Kaufmann}  showed
that $M^1$ is the $\ddbzs$-$\BMO$-$\cp$,
the capacity associated to the space of 
functions of bounded mean oscillation,
and Verdera \cite{Verdera}
established that $M^1_*$
is the $\ddbzs$-$\VMO$-$\cp$, the capacity
associated to the space of functions of vanishing
mean oscillation.  Verdera proved the Vitushkin
theorem for VMO.

The question is, \emph{what do $M^\beta$ and $M_*^\beta$
have to do with the boundary behaviour of analytic functions when
$0<\beta<1$?}.  What is the significance of the
condition
$$ \sum_{n=1}^{\infty} 2^n M^\beta({A_n}\setminus U) < +\infty,$$
when $0<\beta<1$, where 
$U$ is a bounded open subset of $\C$ and $b\in\partial{U}$?

We are considering a local problem, and it is worth noting
that there are several different meanings commonly attached
to the global Lipschitz classes, and the little
Lipschitz classes.  For $0<\alpha<1$, we define
$\Lip\alpha(\R^d)$ to be the space of bounded
functions $f:\R^d\to\C$ that satisfy a Lipschitz-alpha condition:
$$ |f(x)-f(y)| \le \kappa_f|x-y|^\alpha, \ \forall x,y\in\R^d.
$$
We would obtain a locally-equivalent Banach SCS
if we omit the word \lq bounded'.  We would
also obtain a locally-equivalent SCBS if we
just require the Lipschitz condition for
$|x-y|\le1$.  Another locally-equivalent space
is obtained by requiring the Lipschitz condition
with repect to the spherical metric (associated to the
stereographic projection $\mathbb{S}^d\to\R^d$). 
We shall shortly meet another locally-equivalent space,
defined in terms of the Poisson transform.
It makes no difference for our problem which
of these versions is used, and we can exploit this
fact by choosing whatever version is easiest
to use in each context.
For this paper, we define $\lip\alpha$ to be the closure
of $\Test$ in $\Lip\alpha$.  This space is locally-equivalent
to the space of functions that have restriction in
$\lip(\alpha,X)$ for each compact $X\subset\R^d$, but
it has an additional property \lq near $\infty$\rq,
irrelevant for our purposes.

\section{Results}
\subsection{The spaces $T_s$ and $C_s$}
The answer to the problem will not surprise anyone who
has studied the paper \cite{OF-1-reduction}, but may be
regarded as rather strange by others.

The first step in trying to identify the $L$-$F$-$\cp$
for given $L$ and $F$ is based on the principle
that the compact sets $X\subset\R^d$ that have
$(L$-$F$-$\cp)(X)=0$
should be the \emph{sets of removable singularities}
for solutions of $Lf=0$ of class $F$.  This means that
 $(L$-$F$-$\cp)(X)=0$ should be the necessary and sufficient
condition that the restriction map
$$ \{f\in F: Lf=0 \textup{ on }U\}
\to \{f\in F: Lf=0 \textup{ on }U\setminus X\}
$$
be surjective for each open set $U\subset\R^d$.

In \cite[p.140]{OF-1-reduction} it was established (as a
special case of the 1-reduction principle) that
for nonintegral $\beta$ 
the set function $M^\beta$ is zero on the
sets of removable singularities for holomorphic
functions of a Lipschitz class, 
but when $0<\beta<1$ this is a
\emph{negative Lipschitz class}, there denoted $T_{\beta-1}$.

The negative Lipschitz classes can be described in
a number of equivalent ways. In informal terms,
the basic idea is that
the $T_s$ for $s\in\R$ form a one-dimensional
\lq scale' of spaces of distributions,
i.e. a family of spaces totally-ordered under local
inclusion. When $0<s<1$,
$T_s$ is locally-equal to $\Lip s$.
Differentiation takes $T_s$ down to $T_{s-1}$,
and $DT_s$ is locally-equivalent to $T_{s-1}$.
Thus if $s<0$ and $k\in\N$ has $s+k>0$, then
$T_s$ is locally-equal to  
$D^k \Lip(s+k)$.
The elements of $T_s$ having compact support may also
be characterised by the growth of the 
Poisson transform as we approach the plane
from the upper half of 
$\R^3$, or by the growth of the convolution with the
heat kernel.  This idea originated in the work
of Littlewood and Paley and was fully developed
by Taibleson \cite[Chapter 5]{Stein}. The 
Poisson kernel is
$$ P_t(z):= \frac{t}{\pi(t^2+|z|^2)^{\frac32}}, \ (t>0, z\in\C). $$
It is real-analytic in $z$ and $t$, and is
harmonic in $(z,t)$ in the upper half-space
$$ \UH^3:= \C\times(0,+\infty). $$
For a distribution $f\in\mathcal{E}^*:= (C^\infty(\C,\C))^*$
having compact support, the Poisson transform
of $f$ is the convolution
$$ F(z,t):= (P_t*f)(z) $$
where $P_t*f$ denotes the convolution on $\C=\R^2$.
For $s<0$ we say (following \cite{OF-1-reduction})
that $f$ belongs to the \lq negative Lipschitz
space' $T_s$ if 
$$ \|f\|_s:= \sup\{ t^{-s}|F(z,t)|: z\in\C, t>0 \}<+\infty,$$
and belongs to the \lq small negative Lipschitz space
$C_s$ if, in addition,
$$ \lim_{t\downarrow0} t^{-s} \sup\{|F(z,t)|:z\in\C\} = 0. $$

For $s\ge0$, we define $T_s$ and $C_s$ 
by requiring that
for $f\in\mathcal{E}^*$, 
$f\in T_s$ (respectively $C_s$) if and only if all $k$-th order partial
derivatives of $f$ belong to $T_{s-k}$
(respectively $C_{s-k}$)
for each (or, equivalently, for some) integer $k>s$.

The Riesz transforms\footnote{%
Strictly speaking the order $t$
Riesz transform is the operator
$(-\Delta)^{-t/2}$, 
which for $t>0$ is convolution with $c_t|z|^{t-d}$,
for a certain constant $c_t$  
\cite[p 117]{Stein}.
}, convolution with 
$|z|^{t-d}$, map $T_s$ locally into $T_{s+t}$, so behave
like \lq fractional integrals'.

The scale corresponding to the little Lipschitz
class $C_s$ may be described as the closure
of the space 
$\mathcal{D}$ in
$T_s$.  

Delicate questions arise at integral values $s$,
and we shall not consider such $s$ in this paper.

\subsection{Statements}
For an open set $U\subset\C$,  and $s\in\R$, let
$$ A^s(U):= \{f\in C_s: f\textup{ is holomorphic on }U\}, $$
and
$$ B^s(U):= \{f\in T_s: f\textup{ is holomorphic on }U\}. $$

We are interested in the range $-1<s<0$, and for such
$s$ the elements of $A^s(U)$ and $B^s(U)$ are distributions
on $\C$ that may fail to be representable by integration
against a locally-$L^1$ function, so the definition
of continuous point evaluation given above does not
apply.  However, we can make a straightforward adjustment.
We shall prove the following lemma:
\begin{lemma}\label{L:1}
For each $s\in\R$, each open set $U\subset\C$ and each $b\in\C$,
the set $\{f\in A^s(U): f\textup{ is holomorphic on some neighbourhood
of }b\}$ is dense in $A^s(U)$.
\end{lemma}

Here, when we say that the distribution $f$ on $\C$
is holomorphic on an open set $V$, we mean that its
distributional $\bar\partial$-derivative has
support off $V$, i.e.
$$ \left\langle \phi, \frac{\partial f}{\partial\bar z} 
\right\rangle := 
-\left\langle \frac{\partial\phi}{\partial\bar z}, 
f\right\rangle =0 
$$
whenever the test function $\phi$ has support in $V$.
Recall that by Weyl's Lemma this means that the
restriction $f|V$ is represented by an ordinary
holomorphic function, so that it and all its derivatives
have well-defined values throughout $V$.

Let us denote 
$$ A^s_b(U):= 
\{f\in A^s(U): f\textup{ is holomorphic on some neighbourhood
of }b\}.$$

\begin{definition}
We say that $A^s(U)$ admits a continuous point evaluation
at a point $b\in\C$ if the functional $f\mapsto f(b)$
extends continuously from $A^s_b(U)$ to the whole of $A^s(U)$.
\end{definition}

Our main result is this:
\begin{theorem}\label{T:1}
Let $0<\beta<1$ and $s=\beta-1$. 
Let $U\subset\C$ be a bounded open set, and $b\in\partial{U}$.
Then $A^s(U)$ admits a
continuous point evaluation at $b$ if and only if 
$$ \sum_{n=1}^{\infty} 2^n M_*^\beta({A_n}\setminus U) < +\infty.$$
\end{theorem}

\subsection{Weak-star continuous evaluations}
We can also give a result about the big Lipschitz class
$B^s(U)$.  We cannot replace $A^s(U)$ by $B^s(U)$ in Lemma \ref{L:1}
as it stands. However, the $T_s$  spaces are dual spaces,
and so have a weak-star topology (see Subsection
\ref{SS:pf-L-2} for details), and restricting this topology
gives us a second useful topology on $B^s(U)$.
We have the following:
\begin{lemma}\label{L:2}
For each $s\in\R$, each open set $U\subset\C$ and each $b\in\C$,
The set $\{f\in B^s(U): f\textup{ is holomorphic on some neighbourhood
of }b\}$ is weak-star dense in $B^s(U)$.
\end{lemma}

Denoting 
$$ B^s_b(U):= 
\{f\in B^s(U): f\textup{ is holomorphic on some neighbourhood
of }b\},$$
we can then give the following definition:

\begin{definition}
We say that $B^s(U)$ admits a weak-star continuous point evaluation
at a point $b\in\C$ if the functional $f\mapsto f(b)$
extends weak-star continuously from $B^s_b(U)$ to the whole of $B^s(U)$.
\end{definition}

Our result for $B^s(U)$ is this:

\begin{theorem}\label{T:2}
Let $0<\beta<1$ and $s=\beta-1$. 
Let $U\subset\C$ be a bounded open set, and $b\in\partial{U}$.
Then $B^s(U)$ admits a weak-star
continuous point evaluation at $b$ if and only if 
$$ \sum_{n=1}^{\infty} 2^n M^\beta({A_n}\setminus U) < +\infty.$$
\end{theorem}

\subsection{Boundary derivatives}
In the same spirit, we get results about boundary
derivatives. We denote the set of positive integers
by $\N$. 

\begin{theorem}\label{T:3}
Let $0<\beta<1$, $s=\beta-1$, and let $k\in\N$.
Let $U\subset\C$ be a bounded open set, and $b\in\partial{U}$.
Then 
the functional $f\mapsto f^{(k)}(b)$ has a 
continuous extension from $A^s_b(U)$ to
the whole of $A^s(U)$ 
if and only if
$$ \sum_{n=1}^{\infty} 2^{(k+1)n} M_*^\beta({A_n}\setminus U) < +\infty.$$
\end{theorem}

\begin{theorem}\label{T:4}
Let $0<\beta<1$, $s=\beta-1$, and let $k\in\N$.
Let $U\subset\C$ be a bounded open set, and $b\in\partial{U}$.
Then 
the functional $f\mapsto f^{(k)}(b)$ has a weak-star
continuous extension from $B^s_b(U)$ to
the whole of $B^s(U)$ 
if and only if
$$ \sum_{n=1}^{\infty} 2^{(k+1)n} M^\beta({A_n}\setminus U) < +\infty.$$
\end{theorem}

The spaces $A^s(U)$ are not algebras --- essentially
SCS are only algebras when they are locally-included in
$C^0$ ---  so we avoid using the term \emph{derivation},
lest we confuse people.

\subsection{Harmonic functions}
The foregoing results concern objects $f$ that are not \lq proper
functions'. Using the ideas related to 1-reduction,
we may derive a theorem about ordinary harmonic
functions:

For $0<\alpha<1$, let $H^\alpha(U)$ denote
the space of (complex-valued) functions that are harmonic
on $U$ and belong to the little Lipschitz $\alpha$ class on
the closure of $U$ (or, equivalently, have an
extension belonging to the global little Lipschitz
class).  For $b\in\partial{U}$, let
$$H^\alpha_b(U):= 
\{h\in H^\alpha(U): h\textup{ is harmonic on a neigbourhood
of }b\}.
$$

If we denote, as is usual, 
$$\ddz:= \frac1{2}\left(\ddx -i\ddy \right)$$
and
$$\ddbz:= \frac12\left(\ddx + i\ddy\right),$$
then $\Delta = 4\ddz\ddbz$. 

\begin{theorem}\label{T:5}
Let $0<\alpha<1$, let $U\subset\C$ be bounded and open, 
and $b\in\partial{U}$.
Then\\
(1) $H^{\alpha}_b(U)$ is dense in $H^{\alpha}(U)$.
\\
(2) The functional $h\mapsto\dd{h}{z}(b)$ extends continuously
from $H^{\alpha}_b(U)$ to $H^{\alpha}(U)$ if and only if
$$ \sum_{n=1}^{\infty} 2^{n} M^\alpha_*({A_n}\setminus U) < +\infty.$$
\\
(3) The $\C^2$-valued 
function $h\mapsto(\nabla h)(b)$ extends continuously
from $H^{\alpha}_b(U)$ to $H^{\alpha}(U)$ if and only if
$$ \sum_{n=1}^{\infty} 2^{n} M^\alpha_*({A_n}\setminus U) < +\infty.$$
\end{theorem}

\section{Examples}

\subsection{Smooth boundary}
If $U$ is smoothly-bounded, then there are
no continuous point evaluations at any
boundary point on $A^s(U)$ for any $s<0$.
Indeed, if $b$ belongs any to any 
nontrivial continuum $K\subset\C\setminus U$,
then no such continuous point evaluation exists at $b$.  

\subsection{Multiple components}
If $b$ belongs to the boundary of two (or more)
connected components of the open set $U$, then
no such continuous point evaluation exists at $b$.

To see this, note that the assumptions imply that
for all small enough $r$, the circle
$|z-b|=r$ meets the complement of $U$, and this
implies that for large enough $n$,
the $M^\beta$ content of $A_n\setminus U$
is at least $2^{-n\beta}$. Hence the series
in Theorem \ref{T:1} diverges for all $s\in(-1,0)$.

This contrasts with the behaviour found in
\cite{OF-Lord} for ordinary Lipschitz
classes, for which interesting behaviour
is possible at the boundary of Jordan
domains with piecewise-smooth boundary,
and at common boundary points of two
components.

\subsection{Slits}
Let $a_n\downarrow0$, $r_n\downarrow0$ and
$$ a_{n+1}+r_{n+1} < a_n-r_n , \ \forall n\in\N.$$
Then $0$ is a boundary point of the slit domain
$$ U:= \mathring{\mathbb{B}}(0,a_1+r_1) \setminus \bigcup_{n=1}^\infty
[a_n-r_n, a_n+r_n]. $$
For a line segment $I$ of length $d$, we have
$$ M^\beta(I) = M^\beta_*(I) = d^\beta, $$
for $0<\beta<1$.  Then 
for $-1<s<0$, $A^s(U)$ admits a continuous point evaluation at $0$
if and only if $B^s(U)$ admits a weak-star
continuous point evaluation at $0$, and if and only if
\begin{equation}\label{E:slit}
  \sum_{n=1}^\infty \frac{r_n^{s+1}}{a_n} < +\infty. 
\end{equation}
This follows at once from Theorems \ref{T:1}  and
\ref{T:2} in case $a_n=2^{-n}$.  In the general case,
one obtains it by imitating the proofs, using contours that
pass between the slits.  The corresponding condition for
the existence of a $k$-th order continuous point
derivation is
$$
  \sum_{n=1}^\infty \frac{r_n^{(k+1)(s+1)}}{a_n} < +\infty. 
$$
For example, if $a_n=2^{-n}$ and $r_n=4^{-n}$, then
there is a continuous point evaluation at $0$ on $A^s(U)$ if and only if
$s>-\half$.

\subsection{Road-runner sets}
The $M^\beta$ content of a disc and of 
one of its diameters are both fixed multiples of 
$(\textup{radius})^\beta$, and the lower content
is the same, so
the same condition \eqref{E:slit}
is necessary and sufficient for the existence
of a continuous point evaluation on $A^s(U)$ at $0$ on the so-called road-runner
set
$$ U:= 
 \mathring{\mathbb{B}}(0,a_1+r_1) \setminus \bigcup_{n=1}^\infty
\B(a_n,r_n), $$
when the $a_n$ and $r_n$ are as in the last
subsection.

\subsection{Below minus 1}
The $L$-$F$-$\cp$ capacity of a singleton is positive
as soon as there are distributions $f$ of class $F$
having $Lf=0$ on a deleted neighbourhood of $0$.
In the case $L=\ddbzs$, this happens when
the distribution represented by
the $L^1_\loc$ function $\frac1z$ belongs to $F_\loc$.
That explains why, in the case of $L^p$ holomorphic functions,
there is a major transition at $p=2$; the function
$\frac1z$ belongs to $L^p_\loc$ when $1\le p<2$.
(The \lq smoothness' of $L^p(\R^d)$ is
$-d/p$.  This can be extended below $p=1$
by using Hardy spaces $H^p$ instead of $L^p$.)

The \lq delta-function',
$\delta_0$, the unit point mass at the origin,
is the d-bar (distributional) 
derivative of  
$1/z$. 
More precisely
$$  \dd{\frac1{\pi z}}{\bar z} = \delta_0. $$
The Poisson transform of $\delta_0$ is just the
Poisson kernel $t/\pi(|z|^2+t^2)^{\scriptstyle\frac32}$, so
grows no faster than $1/t^2$ as $t\downarrow0$.
Thus $\delta_0$ belongs to $T_{-2}$
and $\frac1z\in T_{-1}$.  

\section{Tools}\label{S:tools}
We abbreviate $\|f\|_{T_s}$ to $\|f\|_s$.
We use $K$ to denote a positive constant which is independent
of everything but the value of the parameter $s$, and
which may be different at each occurrence. 

\subsection{The strong module property}
For a nonnegative integer $k$, and $\phi\in\mathcal{D}$,
we use the notation
$$ N_k(\phi):= d(\phi)^k\cdot \sup|\nabla^k(\phi)|.$$
where $d(\phi)$ denotes the diameter of the support of
$\phi$.  Here, we take the norm $|\nabla^k\phi|$
to be the maximum of all the $k$-th order partial
derivatives of $\phi$.

Note that for $\kappa>0$, $N_k(\kappa\cdot\phi)=\kappa\cdot N_k(\phi)$,
that $$N_0(\phi)\le N_1(\phi)\le N_2(\phi)\le\cdots,$$ and that
(by Leibnitz' formula) 
$$ N_k(\phi\cdot\psi) \le 2^kN_k(\phi)N_k(\psi) $$
whenever $\phi,\psi\in\mathcal{D}$ and $k\in\N$. 

Also, $N_k(\phi)$ is invariant under rescaling:
If $\phi\in\mathcal{D}$, $r>0$, and $\psi(x):=\phi(r\cdot x)$,
then $N_k(\psi)=N_k(\phi)$ for each $k$.

An SCBS $F$ has the \emph{(order $k$-) strong module property}
if there exists $k\in\Z_+$ and $K>0$ such that
$$ \|\phi\cdot f\|_F \le K\cdot N_k(\phi) \cdot\|f\|_F,\ \forall f\in F.$$
Most common SCBS have this property.  Note that
for every SCBS and each $\phi\in\mathcal{D}$
the fact that $F$ is a topological $\mathcal{D}$-module
tells us that there is some constant $K(\phi)$
such that 
$$ \|\phi\cdot f\|_F \le K(\phi) \|f\|_F,\ \forall f\in F.$$
Thus the strong module property amounts to saying that
the least $K(\phi)$ are dominated by some $N_k(\phi)$,
up to a fixed multiplicative constant.

It is readily seen 
that the ordinary Lipschitz spaces
have the order $1$ strong module property.

\subsection{Standard pinchers}
If $b\in\R^d$, a
\emph{standard pincher at $b$}
is a
sequence of nonnegative test functions
$(\phi_n)_n$, 
such that $\phi_n=1$ on a neighbourhood of $b$, 
the diameter of $\spt(\phi_n)$ tends to zero, and
for each $k\in\N$, 
the sequence $(N_k(\phi_n))_n$ is
bounded.

The elements $\phi_n$ of a standard pincher
make the transition from the value $1$ at $b$
down to zero in a reasonably gentle way,
so that the various derivatives are not
greatly larger than they have to be in order
to achieve the transition. 

It is easy to see that such sequences exist.
For instance, they may be constructed in the form
$\phi_n(x) = \psi(n|x-b|)$, where
$\psi:[0,+\infty)\to[0,1]$ is $C^\infty$,
has $\psi=1$ near $0$ and $\psi=0$
off $[0,1]$.

\subsection{The Cauchy transform}
The Cauchy transform is the convolution operator
$$ \mathfrak{C}f:= \frac1{\pi z} * f.
$$
It acts (at least) on distributions having compact support, and
it almost inverts the d-bar operator $\ddbzs$:
$$ \ddbz\mathfrak{C}f = f = \mathfrak{C}\dd{f}{\bar z}, $$
whenever $f\in \mathcal{E}^*$.
Recall from Subsection \ref{SS:1.7} that if $F$ is an SCS
we have the associated spaces $F_\cs$ and $F_\loc$.
The Cauchy transform maps $(T_s)_\cs$ continuously into $T_{s+1}$ and $(C_s)_\cs$ into
$C_{s+1}$, so in combination with $\ddbzs$ it can be used to
relate properties of $T_s$ to properties of $T_{s+1}$.
For our present purposes, this allows us to move from our
spaces of distributions corrresponding to $-1<s<0$
to spaces of ordinary $\Lip(s+1)$ functions.  

The Cauchy kernel $\frac1{\pi\bar z}$ does not belong to
$L^1$, but it does belong to $L^1_\loc$, and indeed
there is a uniform bound on its norm on discs of
fixed radius:
$$  \left\|\frac1{\pi\bar z}\right\|_{L^1(\B(a,r))}
\le 2r, \ \forall a\in\C, \forall r>0. $$
So if $F$ is a SCBS, and translation acts isometrically on $F$,
then
\beq\label{E:C-estimate} \|\mathfrak{C} f\|_F \le d \|f\|_F, \eeq
whenever $f\in F_\cs$  is supported in a disc of radius $d$.

\subsection{Evaluating the Cauchy transform}
The value of $(\Cau f)(b)$ at a point
off the support of the distribution $f$
may be evaluated in the obvious way:
\begin{lemma}\label{L:B}
Let $f\in\mathcal{E}^*$, and $b\in\C\setminus\spt(f)$. Let
$\chi\in\Test$ be any test function having
$\chi(z)=1/(z-b)$ near $\spt(f)$. Then
$$ \Cau(f)(b) = \left\langle \frac{\chi}{\pi},f\right\rangle.$$
\end{lemma}
\begin{proof}
We take $b=0$, without loss in generality.

For any $\psi\in\Test$ with $\int\psi dm=1$ and
$\spt(\psi)\cap\spt(f)=\emptyset$, we have
$$ \langle\psi,\Cau(f)\rangle
=-\left\langle \psi*\frac1{\pi z}, f\right\rangle.$$

Let $(\phi_n)_n$ be a standard pincher at $0$
and take 
$$\psi_n=\frac{\phi_n}{\int\phi_ndm}. $$
Then 
$\psi_n*\frac1{\pi z}\to\frac1{\pi z}=\chi/\pi$
in $C^\infty$ topology on a neighbourhood of
$\spt(f)$, so 
$$\left\langle \frac{\chi}{\pi},f\right\rangle =
\lim_n 
\left\langle \psi_n*\frac1{\pi z}, f  \right\rangle 
=
-\lim_n 
\langle \psi_n, \Cau(f)  \rangle 
=\Cau(f)(0).$$
\end{proof}

\subsection{The Vitushkin localization operator}\label{SS:proof-lemma-1}
The Vitushkin localization operator is defined by
$$ \mathfrak{T}_\phi(f):=  \mathfrak{C}\left(\phi\cdot\
\dd{f}{\bar z}
\right). $$
Here $\phi\in\mathcal{D}$ and $f\in\mathcal{D}^*$.

In view of the distributional equation
$$ \ddbz \Vit_\phi(f) = \phi\cdot \dd{f}{\bar z}, $$
$\Vit_\phi(f)$ is holomorphic wherever $f$ is 
holomorphic and off the support of $\phi$.

It was established in \cite{OF-T-invariance}
(using soft general arguments)
that whenever $F$ is an SCS, $\Vit_\phi$ maps $F$ continuously
into $F_\loc$, and that when $F$ is an SCBS,
we actually get a continuous map into the
Banach subspace  $F_\infty\subset F_\loc$ normed by
$$ \|f\|_{F_\infty}:= \sup\{ \|f|B\|_{F(B)}:B\textup{
is a ball of radius }1 \}, $$
where $f|B$ denotes the restriction coset
$f+J(F,B)$, with $J(F,B)$ equal to the space
of all elements $g\in F$ that vanish near $B$, and
the $F(B)$ norm of a restriction is the infimum
of the $F$ norms of all its extensions in $F$, i.e.
$$ \|f|B\|_{F(B)}:= \inf\{ \|h\|_F: h\in F, h-f =0 \textup{ near }B\}. $$

When $F$ is an SCBS with the strong module property
(of order $k$),
and translation acts isometrically on $F$,
the identity
$$ \mathfrak{T}_\phi (f) = \phi\cdot f - 
\mathfrak{C}\left(\dd{\phi}{\bar z}
\cdot f \right) $$
together with equation \eqref{E:C-estimate}
yields the more precise estimate
\beq\label{E:2} \|\mathfrak{T}_\phi (f)\|_F \le
K N_{k+1}(\phi)\cdot \|f\|_F, \ \forall \phi\in\mathcal{D},
\forall f\in F.
\eeq

\subsection{Our spaces $T_s$}
We denote the $(T_s)_\infty$ norm of a distribution
$f$ by $\|f\|_{s,\infty}$.  Notice that if the support
of $f$ has diameter at most $1$, then $\|f\|_s$
and $\|f\|_{s,\infty}$ are comparable, i.e.
stay within constant multiplicative bounds
of one another.  In fact, using only
the translation-invariance of the norm
and the (ordinary) $\mathcal{D}$-module
property, it is easy to see that, for such $f$, 
$$ \|f\|_{s,\infty} \le \|f\|_s \le K\|f\|_{s,\infty},$$
where $K$ is independent of $f$.  

%

\begin{lemma}\label{L:2.5}\label{L:4}
Let $k\in\N$ and $-k-1<s<-k$. Then $(T_s)_\infty$ has the 
order $(k+2)$ strong module property,
and in fact 
$$ 
\|\phi\cdot f\|_s \le K\cdot N_{k+2}(\phi)\cdot \|f\|_s, $$
for all $\phi\in\mathcal{D}$ with $d(\phi)\le1$
and all $f\in T_s$,
where $K>0$ is independent of $\phi$ and $f$.
\end{lemma}
\begin{proof}
We use induction on $k$, \emph{starting
with $k=-1$}.

For $0<s<1$, $\Lip(s)$ has the strong module
property of order $1$, and since $T_s$ is
locally-equal to $\Lip(s)$, we have the result
in this case.

Now suppose it holds for some $k$, and
fix $s\in(-k-2,-k-1)$.

It suffices to prove the estimate for $\phi$
supported in $\B(0,2)$, and multiplying
$f$ by a fixed test function $\psi$ that equals $1$
on $\B(0,2)$, we may assume that $f$ has
compact support (without changing $\phi\cdot f$ or
increasing $\|f\|_s$ by more than a fixed constant
that depends only on $\psi$ and $s$). Then
$g:=\mathfrak{C}f\in T_{s+1}$ has 
$\dd{g}{\bar z}=f$ and $\|g\|_{s+1}\le K\|f\|_s$. 
	Also $\ddbzs$ maps $T_{s+1}$ continuously
	into $T_s$, and $\ddbzs\Cau(\phi\cdot f)=\phi\cdot f$,
	so using \eqref{E:2} with $k$ replaced by $k+2$, 
we have
$$\|\phi\cdot f\|_s \le 
K\|\Cau(\phi\cdot f)\|_{s+1}
= K \|\Vit_\phi(g)\|_{s+1} 
\le K\cdot N_{k+3}(\phi)\|g\|_{s+1,\infty},
$$
since $(T_{s+1})_\infty$ has the 
strong module property of order $k+2$, by the induction hypothesis.
Then
$$\|\phi\cdot f\|_s \le 
K\cdot N_{k+3}(\phi)\cdot\|g\|_{s+1}
\le
K\cdot N_{k+3}(\phi)\cdot \|f\|_s
$$
Hence the result holds for $k+1$, completing
the induction step.
\end{proof}

%

\begin{remark}
A similar result holds for all s, but for
positive $s$ one has to
replace $\|g\|$ by $\|g-p\|$, where $p$ is the 
degree $\lfloor s\rfloor$ Taylor polynomial of
$g$ about $a$.  
\end{remark}


\subsection{The $C_s$ norm on small discs}
Since $\Cau$ maps $(C_s)_\cs$
into $C_{s+1}$,
induction also gives the following:

\begin{lemma}\label{L:3} Let $s<0$, 
$f\in C_s$ and $a\in\C$.
Then for each 
$\epsilon>0$ 
there exists $r>0$ and
$g\in C_s$ such that $f=g$ on $\mathring\B(a,r)$
and $\|g\|_{T_s}<\epsilon$.
\qed\end{lemma}

\begin{remark}
We note that since 
$\Vit_\phi(f)$ is holomorphic off the support of $\phi$
and has a zero at $\infty$, $\Vit_\phi$ maps $C_s$ into $C_s$.
\end{remark}

\subsection{Estimate for $\langle\phi,f\rangle$}
The strong module property gives an estimate
for the action of $f\in F$ on a given
$\phi\in\Test$:
\begin{lemma}\label{L:general-estimate}
Suppose $F$ is an SCBS with the order $k$
strong module property. Then for each compact
$X\subset\C$, there exists $K>0$ such that
$$|\langle\phi,f\rangle| \le K\cdot
N_k(\phi)\cdot\|f\|_F, $$
whenever $\phi\in\Test$,
$f\in F$ and $\spt(\phi\cdot f)\subset X$.
\end{lemma}
\begin{proof}
Fix $\chi\in\Test$ with $\chi=1$ near $X$.
Then since $f\mapsto\langle\chi, f\rangle$ is continuous
there exists $K>0$ such that 
$$|\langle\chi,f\rangle| \le K\cdot\|f\|_F. $$
Thus
$$|\langle\phi,f\rangle| = |\langle\chi,\phi\cdot f|\le K\cdot
N_k(\phi)\cdot\|f\|_F. $$
\end{proof}

Putting it another way, the order $k$ strong module
property says that the operator $f\mapsto \phi\cdot f$
on $F$ has operator norm dominated by $N_k(\phi)$,
and this last lemma says that the functional
$f\mapsto \langle\phi,f\rangle$ has $F(X)^*$
norm dominated by $N_k(\phi)$.  

It turns out that we can improve substantially
on this estimate, for our particular spaces.
The trick is to pay close attention to
the support of $\phi\cdot f$.



\subsection{Scaling: better estimate}
The action of an affine map $A:\R^d\to\R^d$ on
distributions is defined by
$$ \langle\phi,f\circ A\rangle := 
 |A|^{-1}\langle\phi\circ A^{-1},f\rangle,
\ \forall\phi\in\Test.$$ 
In the case of a dilation $A(z)=r\cdot z$ on $\C$,
this means that
$$ 
 \langle\phi,f\rangle
=r^2\langle\phi\circ A,f\circ A\rangle. 
$$ 
Taking into account the fact that $N_k(\phi)=N_k(\phi\circ A)$
whereas the identity
$$ (P_t*f)(z) = (P_{t/r}*(f\circ A))(rz) $$
gives 
$$\|f\circ A\|_s = r^s \|f\|_s,$$ 
we obtain:
\begin{lemma}\label{L:extra-d}
Let $-2<s<0$. Then 
$$|\langle\phi,f\rangle| \le 
K\cdot N_{3}(\phi)\cdot \|f\|_s
\cdot r^{s+2},
$$ 
whenever $\phi\in\Test$ and $\spt(\phi\cdot f)\in\B(0,r)$.
\qed
\end{lemma}

Note that for $f\in C_s$, the norm of $f$ in $(T_s)_{\B(a,r)}$
tends to zero as $r\downarrow0$, so we can replace
the constant $K$ in the estimate by $\eta(r)$,
where $\eta$ depends on $f$, and $\eta(r)\to0$ as
$r\downarrow0$.

\subsection{Hausdorff content estimate}
Next, using a covering argument, we can bootstrap
the estimate to:
\begin{lemma}\label{L:M-beta}
Let $-2<s<0$. Then 
$$|\langle\phi,f\rangle| \le K\cdot N_3(\phi)\cdot \|f\|_s
\cdot M^{s+2}(\spt(\phi\cdot f)),
$$ 
whenever $\phi\in\Test$ and $f\in T_s$.
\qed
\end{lemma}

Before giving the proof of this lemma,
we need some preliminaries.

A \emph{closed dyadic square} is a set of the form
$I_{m,n}\times I_{r,n}$ (for integers $n,m,r$) where
$$ I_{m,n}:= \left\{x\in\R: \frac{m}{2^n}\le x\le
\frac{m+1}{2^n}\right\}.$$
Let $\mathcal{S}_2$ denote the family of 
closed dyadic squares.  For $\beta>0$, the \emph{
$\beta$-dimensional dyadic content}
of a set $E\subset\R^2$ is 
$$ M^\beta_2(E):= \inf\left\{
\sum_{n=1}^\infty
\textup{side}(S_n)^\beta:
E\subset \bigcup_{n=1}^\infty S_n,
\textup{ and } S_n\in\mathcal{S}_2
\right\}.
$$
This content is comparable to $M^\beta$:
$$ M^\beta(E)\le 2^{\beta/2}M_2^\beta(E),\quad
M_2^\beta(E)\le 2^{\beta+2} M^\beta(E), $$
for all bounded sets $E$.

Thus it suffices to prove Lemma \ref{L:M-beta} with
$M^{s+2}$ replaced by $M_2^{s+2}$ in the statement.

Also, since both sides change by
a factor $r^2$ when $f$ and $\phi$
are replaced by $f(r\cdot)$ and $\phi(r\cdot)$,
it suffices to consider the case
when $E:=\spt(\phi\cdot f)$ has diameter
at most $1$.  Given $\epsilon>0$, we may cover 
such an $E$ by a countable sequence
$(S_n)_n$ of dyadic squares, of side at
most $1$, with
$$ \sum_{n=1}^\infty \side(S_n)^\beta < M_2^\beta+\epsilon.$$

We now state the key partition-of-identity lemma:

\begin{lemma}\label{L:partition}
Let $k\in\N$ be given.
There exist positive constants $K$ and $\Lambda$
such that
whenever
$E\in\R^2$ is compact, and 
$$ E \subset  
\bigcup_{n=1}^\infty S_n,$$
where the $S_n$ are dyadic squares of side
at most $1$, there exists a sequence of
test functions $(\phi_n)_n$ 
such that 
\\ (1) $\phi_n=0$ except for finitely many $n$;
\\ (2) $\sum_n\phi_n=1$ on a neighbourhood of $E$;
\\ (3) $N_k(\phi_n)\le K$ for all $n$,
and 
\\ (4) $\spt(\phi_n)\subset \Lambda S_n$.
\end{lemma}

Here $\Lambda S$ denotes the square with the same
centre as $S$ and $\Lambda$ times the side.
We shall show that $\Lambda$ may be taken equal to $5$,
although we do not claim this is sharp.
\begin{proof}
Rearrange the $S_n$ in nonincreasing order of size.
The interiors of the $\frac54 S_n$ form an open covering
of $E$, so we may select a finite subcover,
$$\mathcal{F}:= 
\left\{ \textstyle\frac54 \mathring{S_n}: 1\le n\le N\right\}.$$ 
Remove all squares from the sequence $(S_n)_n$ that are contained
in $(\frac54S_1) \setminus S_1$, re-number the remaining squares,
and adjust $n$; then remove all 
in $(\frac54S_2) \setminus S_2$, and so on.
Now
no element $S_n\in\mathcal{F}$ is contained in
any square \lq cordon' $(\frac54S_m)\setminus S_m$.

Group the squares of $\mathcal{F}$ into generations
$$ \mathcal{G}_m:= \left\{
S\in\mathcal{F}: \side(S)=2^{-m}
\right\} $$
for $m=0,1,2,\ldots$. 

Each (finite) generation $\mathcal{G}_m$ forms
part of the tesselation $\mathcal{T}_m$ of the whole plane by
dyadic squares of side $2^{-m}$.  We can construct
a uniform partition of unity on the whole plane subordinate to the
covering by the open squares $5\mathring{S}$, with
$S\in\mathcal{T}_m$, as follows:  

Choose $\rho\in C^{\infty}([0,+\infty))$ such that
$\rho$ is nonincreasing, $\rho(r)=1$ for $0\le r\le\frac58$
and $\rho(r)=0$ for $r\ge\frac34$. Then define
$$\theta(x,y):= \rho(x)\rho(y). $$
For a dyadic square $S$ having centre $(a,b)$
and side $1$, define 
$\theta_S(x,y):= \theta(x-a,y-b)$, and
$$\tau:= \sum_{S\in\mathcal{T}_1} \theta_S. $$
Then $\theta_S=1$ on $\frac54S$ and is supported
on $\frac32S$, so that $1\le\tau\le4$.  
Let
$$\psi_S:= \frac{\theta_S}{\tau}.  $$
Then the test functions $\psi_S$, for $S\in\mathcal{T}_1$
 form
a partition of unity, and $c_k:=N_k(\psi_S)$
is independent of $S$.  (This partition is invariant
under translation by Gaussian integers.)

For general $m\in\N$, and a dyadic square
$S$ of side $2^{-m}$, define
$\psi_S(z):= \psi_{2^mS}(2^mz)$.  Then 
the $\psi_S$, for $S\in\mathcal{T}_m$
also form a nonnegative smooth partition of unity, $N_k(\psi_S)=c_k$
is independent of $S$ (and $m$), 
the support of $\psi_S$ is contained in
$\frac32S$, hence at most $4$
$\psi_S$ are nonzero at any given point.  

Note that 
$$|\nabla^k\psi_S|\le \left(\textstyle\diam\frac32S\right)^k c_k = 
\left(\textstyle\frac3{\sqrt2}\right)^k \cdot
(\side S)^k \cdot c_k $$
for each $k\in\N$.

For a dyadic square $S$, let $S^+$ denote
the set of $9$ dyadic squares of the same size
that meet $S$.  For any family $\mathcal{H}$
of dyadic squares, let
$$ \mathcal{H}^+:= \bigcup\{ S^+: S\in\mathcal{H}\}.$$
Thus $S^{++}:=(S^{+})^{+}$ is the family of $25$ squares,
consisting of $S$, the $8$ other dyadic squares 
of the same size that meet $S$,
and the $16$ other dyadic squares of the same size
that meet at least one of those $8$ squares.
Observe that the smooth function
$$ \sum_{T\in S^{+}} \psi_T $$
is supported in $\bigcup S^{++}=5S$ and has sum
identically $1$ on $\frac32S$.  

\smallskip
We now proceed to construct the desired collection
of functions $(\phi_n)$.  

Let 
$$\sigma_m:=\sum_{S\in \mathcal{G}_m^+} \psi_S.$$
Then $\sigma_m$ is supported in $\bigcup\mathcal{G}_m^{++}$
and 
$$ \sigma_m = 1 \textup{ on }
K_m:= \bigcup_{S\in\mathcal{G}_m} \textstyle\frac32S. $$ 
Since at most $4$ $\psi_S$ are nonzero at any one point, 
we have 
\begin{equation}\label{E:sigma-bound}
|\nabla^k\sigma_m| \le
4c_k\left(\textstyle\frac3{\sqrt2}\right)^k\cdot 2^{km}, \ \forall k\in\N. 
\end{equation}

Now take the squares $S\in\mathcal{G}_0^+$,
and allocate each one to a \lq nearest' square 
$n(S)\in\mathcal{G}_0$
so that: 
\\
(1) if $S\in\mathcal{G}_0$, take $n(S)=S$;
\\
(2) if $S\not\in\mathcal{G}_0$, pick $n(S)$
with $S\in n(S)^+$.  (There may be up to
eight ways to pick $n(S)$. It does not matter
which you choose.)

Next, let 
$$\phi_T:= \sum_{n(S)=T} \psi_S, \ \forall T\in\mathcal{G}_0. $$
Then each $\phi_T$ is supported on $5T$, and
$$ \sigma_0 = \sum_{T\in\mathcal{G}_0} \phi_T. $$
Since the sum defining $\phi_T$ has at most
$9$ terms, and its support has diameter at most
$5$ times that of $T$, we have 
$$N_k(\phi_T) \le 9\cdot5^k\cdot c_k.$$
Let $\tau_0:= \sigma_0$.

Next, consider $\mathcal{G}_1$. As before,
allocate each square $S\in\mathcal{G}_1^+$
to a nearest square $n(S)\in\mathcal{G}_1$, but this
time let
$$\phi_T:= (1-\tau_0)\sum_{n(S)=T} \psi_S, \ \forall T\in\mathcal{G}_1. $$
Then
$$ \sum_{T\in\mathcal{G}_1} \phi_T
= (1-\tau_0)\sigma_1.
$$
and
$$ \tau_1:= \tau_0 + (1-\tau_0)\sigma_1 $$
is supported in $\bigcup(\mathcal{G}_0^{++}\cup\mathcal{G}_1^{++})$
and is identically 
equal to $1$ on
$K_0\cup K_1$.

Continuing in this way, for $m\ge1$ we 
allocate each square $S\in\mathcal{G}_{m+1}^+$
to a nearest square $n(S)\in\mathcal{G}_{m+1}$, 
and let
$$\phi_T:= (1-\tau_m)\sum_{n(S)=T} \psi_S, \ \forall T\in\mathcal{G}_{m+1}, $$
and
$$ \tau_{m+1} := 
 \tau_{m}+ (1-\tau_m)\sigma_{m+1}.
$$
Then 
$$ \sum_{T\in\mathcal{G}_{m+1}} \phi_T = (1-\tau_m)\sigma_{m+1} $$
and $\tau_{m+1}=1$ on $K_0\cup\cdots\cup K_{m+1}$.

When we have worked through all the nonempty
generations $\mathcal{G}_m$, we will have defined
$\phi_{S_n}$ for each $S_n$, and 
(renaming $\phi_{S_n}$ as $\phi_n$)
we have $\sum_n\phi_n =1$ on the union
of all the $\frac32S_n$, and hence on a neighbourhood of
$E$.  Since $\phi_n$ is supported on $5S_n$,
it remains to prove the estimate (3) of the statement,
i.e. to prove that $\sup_nN_k(\phi_n)<+\infty$.

This amounts to showing that there is
a constant $K>0$ (depending on $k$) such that
for $0\le m\in\Z$ and $S\in\mathcal{G}_m$,
we have
$$  | \nabla^k \phi_S | \le K 2^{km}. $$
To do this, we start by
proving that for each $k$ there exists $C=C_k>0$ such that
\begin{equation}\label{E:tau-bound}
|\nabla^k \tau_m| \le  
C\cdot2^{km}, \ \forall k\in\N. 
\end{equation}
for all $m\ge0$.

To see this, we use induction on $k$ 
and on $m$, and
the identity
\begin{equation}\label{E:identity}
 \tau_{m+1} = \sigma_{m+1} + (1-\sigma_{m+1})\tau_m, 
\end{equation}
together with the bound \eqref{E:sigma-bound}.

Let us write
$$d_k:= 4c_k\left(\textstyle\frac3{\sqrt2}\right)^k,$$
so that \eqref{E:sigma-bound} becomes
$ |\nabla^k\sigma_m| \le d_k 2^{km}$. 
Since $\tau_0=\sigma_0$, then 
for any $k$, we 
know that \eqref{E:tau-bound}
holds for $m=0$ as long as $C$ is at least $d_k$.

Take the case $k=1$, and proceed by induction on
$m$. If \eqref{E:tau-bound} holds for $k=1$ and some $m$, then the identity
\eqref{E:identity} gives
$$ |\nabla\tau_{m+1}| \le d_12^{m+1}
 + |\nabla\tau_m| + |\nabla\sigma_{m+1}|
\le \left(d_1 + \frac{C}{2} + d_1\right)2^{m+1}. $$
Thus we get \eqref{E:tau-bound} with $m$
replaced by $m+1$, as long as $C\ge4d_1$.  
This proves the case $k=1$, with $C_1=4d_1$.

Now suppose that $k>1$, and we have \eqref{E:tau-bound}
with $k$ replaced by any number $r$ from
$1$ to $k-1$ and $C$ replaced by some $C_r$.  
We
proceed by induction on $m$. 
We have the case $m=0$, with any constant $C\ge d_k$.
Suppose we have the case $m$, with a constant $C$.

Using the identity, we can estimate $|\nabla^k\tau_{m+1}|$
by
$$ d_{k}2^{k(m+1)} + C\cdot2^{km} + \sum_{j=1}^k
\binom{k}{j}|\nabla^j\sigma_{m+1}|\cdot|\nabla^{k-j}\tau_m| 
.$$
This is no greater than
$$ \left( \frac{C}{2^k} + R\right) \cdot 2^{k(m+1)}, $$
where $R$ is an expression involving $d_1$,$\ldots$,$d_k$
and $C_1$,$\ldots$,$C_{k-1}$.  So as long as
$C>2R$, we get \eqref{E:tau-bound} with $m$
replaced by $m+1$, and the induction goes through.

So we have \eqref{E:tau-bound} for all $k$ and
$m$.  It follows easily that for some $C>0$
(depending on $k$) and for each $S\in\mathcal{G}_m^+$
we have 
$$| \nabla^k (1-\tau_m)\cdot \psi_S | \le C\cdot2^{km},$$
and this gives 
$$| \nabla^k \phi_S | \le 9C\cdot2^{km}$$
whenever $S\in\mathcal{G}_m$, as required.
\end{proof}

\begin{proof}[Proof of Lemma \ref{L:M-beta}]
With $E=\spt(\phi\cdot f)$, take 
the partition of the identity $(\phi_n)$ constructed
in Lemma \ref{L:partition},
and note that $\phi = \sum_n\phi\cdot\phi_n$
on a neighbourhood of $E$. 
Thus
$$ \langle \phi, f \rangle =
\sum_{n=1}^N \langle  \phi\cdot\phi_n, f\rangle. $$
Now apply Lemma \ref{L:extra-d} with $\phi$
replace by $\phi\cdot\phi_n$.
The fact that $N_k$ is submultiplicative
implies that $N_3(\phi\cdot\phi_n)\le KN_3(\phi)$, so we get
the stated result at once.
\end{proof} 

Further, using the remark about $\eta(r)\downarrow0$, we
get a stronger statement for elements of $C_s$:

\begin{lemma}\label{L:M-beta-*}
Let $-2<s<0$. Then 
$$|\langle\phi,f\rangle| \le K \cdot N_3(\phi)\cdot \|f\|_s
\cdot M^{s+2}_*(\spt(\phi\cdot f),
$$ 
whenever $\phi\in\Test$ and $f\in C_s$.
\qed
\end{lemma}

\section{Proofs of preliminary lemmas}\label{S:proofs}
\subsection{Proof of Lemma \ref{L:1}}
\begin{proof}
Fix $f\in A^s(U)$. Take some $\psi\in\mathcal{D}$
having $\psi=1$ near $\overline{U}$.  Then
$f_1:= \psi\cdot f\in A^s(U)$, so we may write
$f= f_1 + f_2$, where $f_2\in C_s$ vanishes near 
$\overline{U}$, and hence is holomorphic near
$b$. So it remains to show that we can
approximate $f_1$ by elements of $A^s_b(U)$.

Now $f_1$ has compact support. 
Take a
standard pincher 
$(\phi_n)_n$ 
at $b$.

Take $g_n:= T_{\phi_n}(f_1)$. Then $\|g_n\|_s\le K\|f_1\|_s$.
Since $T_{\phi_n}$ depends only on the restriction
of $f_1$ to the support of $\phi_n$, an application of
Lemmas \ref{L:3} and \ref{L:4} shows that $\|g_n\|_s\to0$ as
$n\uparrow\infty$. Thus $f_1-g_n\to f_1$ in $T_s$ norm.
Finally,
$$ \ddbz(f_1-g_n) = (1- \phi_n)
\dd{f_1}{\bar z}, $$
so $f_1-g_n$ is holomorphic on a neighbourhood of $b$,
and so belongs to $A^s_b(U)$. 
\end{proof}

\subsection{Proof of Lemma \ref{L:2}}
\label{SS:pf-L-2}

First, we have to explain the 
weak-star topology in question, by specifying
a specific predual for $T_s$.

The fact is that $T_s$ is essentially
the double dual of $C_s$. More, it is 
a \emph{concrete dual}:  An SCS $F$ is
called \emph{small} if it is the closure of
$\mathcal{D}$. If $F$ is a small SCS, then 
its dual $F^*$ is naturally isomorphic to an SCS,
where the isomorphism is the restriction map
$L\mapsto L|\mathcal{D}$.  We call this SCS
the concrete dual of $F$, and denote it by
the same symbol $F^*$.
Also $F_\loc$ and $F_\cs$ are also small, 
$$ (F_\loc)^* = (F^*)_\cs,$$
$$ (F_\cs)^* = (F^*)_\loc,$$
and so
$$ (F^*) \loceq (F_\loc)^* \loceq (F_\cs)^*
.
$$

In the case of $C_s$, for $0<s<1$, the
concrete dual $C_s^*$ is also small, and we have
$$ (C_s^*)^* \loceq T_s. $$
This fact is basically due to Sherbert,
who observed the isomorphism
$$ \lip(\alpha,K)^{**} = \Lip(\alpha,K) $$
for all compact metric spaces $K$.  The key to
this is the fact that for each $L\in \lip(\alpha,K)^*$
that annihilates constants there exists a measure
$\mu$ on $K\times K$, having no mass
on the diagonal, such that
$$ Lf = 
\int_{K\times K}
\frac{f(x)-f(y)}{\dist(x,y)^\alpha}
d\mu(x,y), $$
whenever $f\in\lip(\alpha,K)$. 
In particular, each point-mass at an
off-diagonal point $(z,w)\in\C\times\C$ gives
an element of the dual of $\lip\alpha(\C)$:
$$ Q(z,w)(f):= 
\frac{f(z)-f(w)}{|z-w|^\alpha},
\ \forall f\in\lip\alpha. $$
This might lead one to suspect that the dual
is non-separable, but the norm
topology on these functionals is not discrete.
In fact, the map
$$ P:\C^2\setminus\textup{diagonal} \to \lip\alpha^* $$
is continuous, and indeed locally H\"older-continuous:
one may show that
$$ \|Q(z,w) - Q(z',w)\|_{\lip\alpha^*}
\le \frac{4 |z-z'|^\alpha}{|z-w|} $$
whenever $|z-z'|<\half|z-w|$.
Hence the functional $L$ on $\lip\alpha_\loc$ represented by
a given measure $\mu$ may be approximated in the
dual norm by finite linear combinations of elements
from
$$\mathcal{P}:= \{ Q(z,w): z\not=w \}.
$$
By smearing the point masses,
each functional $Q(z,w)$ may be approximated in the dual norm by 
functionals $\int_\C Q(z+\zeta,w+\zeta)\phi(\zeta)dm(\zeta)$,
where $\phi\in\mathcal{D}$ has $\int\phi\,dm=1$, 
which send an element $f\in\lip(\alpha)$ to
$$
\int_\C Q(z+\zeta,w+\zeta)(f) \cdot \phi(\zeta) \,dm(\zeta) 
$$$$= 
\int_\C f(\omega) \left(
\frac{\phi(\omega-z)-\phi(\omega-w)}{|z-w|^\alpha}
\right)
dm(\omega),
$$
and the function
$$
\omega\mapsto  
\frac{\phi(\omega-z)-\phi(\omega-w)}{|z-w|^\alpha}
$$
is a test function, 
so the functional $L$ may be approximated by test functions. 
Thus $\lip\alpha_\loc^*$ has a concrete dual,
and by Sherbert's result this can only
be $\Lip\alpha_{cs}$.

Moreover, it follows that a sequence in $\Lip\alpha$
is weak-star convergent to zero if and only if it is
bounded in $\Lip\alpha$ norm and 
converges pointwise to zero on the span of $\mathcal{P}$.
So in fact it suffices to show that it is
bounded in norm and converges pointwise
on $\C$.  But we already know
that if $(\phi_n)$ is a standard pincher,
then, for $f\in\Lip\alpha$, $T_{\phi_n}f$
is bounded in $\Lip\alpha$ norm and
converges uniformly to zero, hence
we conclude that $T_{\phi_n}f$
is weak-star convergent to zero.
This proves the lemma in case $0<s<1$.

For other nonintegral $s$, we obtain it by
applying the Fundamental Theorem of Calculus.
In particular, for the case of immediate interest,
$-1<s<0$, we have that
$$ ((C_s)_\loc)^{**} \loceq ((DC_{s+1})_\loc)^{**}
\loceq \sint T_{s+1} \loceq T_s, $$
so to show that, for $f\in T_s$, the sequence
$(T_{\phi_n}f)$ converges weak-star in $T_s$, 
it suffices to show that
$(\mathfrak{C}T_{\phi_n}f)$ converges weak-star
in $T_{s+1}$. 
We may assume that
$f$ has compact support, since 
$(T_{\phi_n}f)$
depends only on the restriction of $f$
to a neighbourhood of $\spt\phi$, and then
taking $g=\mathfrak{C}f\in T_{s+1}$, 
it suffices to show that
$\mathfrak{C}T_{\phi_n}\dd{g}{\bar z}$ converges
weak-star to zero.  But
$$\mathfrak{C}T_{\phi_n}\dd{g}{\bar z} =
\mathfrak{C}^2 \left(
\phi_n\cdot \frac{\partial^2 g}{\partial \bar z^2}
\right),$$
so we are just dealing with the equivalent
of $\Vit_\phi$ for the d-bar-squared operator
instead of the d-bar operator, so it is
bounded on $\Lip\alpha$ and on $C^0$, 
independently of $n$, and thus we
have the desired weak-star convergence. 

\begin{remark}
We expect that the argument of Subsection \ref{SS:pf-L-2}
may be used more generally, i.e
we \emph{conjecture} the following:

\textsl{Let $F$ be a small SCBS, such that $F^*$ is also small,
$F^{**}\locin C^0$, and the span
of the point evaluations is dense in $F^*$,
and $F^{**}$ has the strong module property.
Then whenever $(\phi_n)$ is a standard pincher,
and $L$ is an elliptic operator with smooth
coefficients, 
$$ L^{-1} (\phi_n\cdot Lf) \to 0 \textup{ weak-star }
\ \forall f\in F^{**}. $$
}

Here, $L^{-1}$ denotes some suitably-chosen parametrix for $L$.
\end{remark}

\section{Proofs of Theorems}

\subsection{Proof of Theorem \ref{T:1}}

We fix $\beta\in(0,1)$ and $s=\beta-1$,
and without loss in generality we assume 
that the boundary point $b=0$.

\medskip
First, consider the \lq only if' direction.
Suppose the series diverges:
$$ \sum_{n=1}^\infty 2^n M^\beta_*(A_n\setminus U)=+\infty.$$
We wish to show that there exist 
$f\in A^s_0(U)$ having $\|f\|_s\le1$ and
$|f(0)|$ arbitrarily large.

Since $M^{\beta}_*$ is subadditive, there exists 
at least one of the four right-angle sectors
$$ S_r:= \left\{z\in\C: |\textup{arg}(i^rz)|<\frac{\pi}4\right\} $$
(for $r\in\{0,1,2,3\}$)
such that 
$$ \sum_{n=1}^\infty 2^n M^\beta_*((S_r\cap A_n)\setminus U)=+\infty.$$
We may assume that this happens for $r=0$,
and we may assume further that $U$ contains
the whole complement of $S_0$ and the whole
exterior of the unit disc.  So we may select
closed sets $E_n\subset S_0\cap A_n$ such that
$U\cap E_n=\emptyset$ and
$$ \sum_{n=1}^\infty 2^n M^\beta_*(E_n)=+\infty.$$
We may select numbers $\lambda_n>0$ such that the individual terms
\\$\lambda_n 2^n M^\beta_*(E_n)\le 1$, and yet
$$ \sum_{n=1}^\infty \lambda_n 2^n M^\beta_*(E_n)=+\infty.$$

For each $n$, by Frostman's Lemma, we may select a
positive Radon measure supported on $E_n$ such that
(1) $\mu_n(\B(a,r))\le r^\beta$ for all $a\in\C$
and all $r>0$ (i.e. $\mu_n$
\lq has growth $\beta$\rq), (2) the total variation $\|\mu_n\|\ge
K\cdot M^\beta_*(E_n)$, and (3)
$\mu_n(\B(a,r))/r^\beta\to0$ uniformly in
$a$ as $r\downarrow0$.  Then taking $r\mapsto h(r)$ to be the
upper concave envelope of $r\mapsto\sup_a 
\mu_n(\B(a,r)$ on $[0,+\infty)$,
we have $\|\mu_n\|\le M_h(E_n) \le M^\beta_*(E_n)$, so 
$\lambda_n2^n\|\mu_n\|\le1$ and
$$ \sum_{n=1}^\infty \lambda_n 2^n \|\mu_n\|=+\infty.$$

Let $h_n:=\lambda_n\Cau(\mu_n)$. Then $h_n\in C_s$,
$h_n$ is holomorphic off $\spt(\mu_n)$, hence
$h_n\in A^s_0(U)$.  Also $\Re(h_n(0))\ge \lambda_n\frac{2^n}{\sqrt2}\|\mu_n\|$.
Hence 
$$ \left|\sum_{n=1}^N h_n(0)\right| \ge
\frac1{\sqrt2} \sum_{n=1}^N \lambda_n 2^n \|\mu_n\| \to +\infty $$
as $N\uparrow\infty$.  So now it suffices to show
that $f_N:= \sum_{n=1}^N h_n$ is bounded in
$T_s$ norm, independently of $N\in\N$.

For this, it suffices to show that the $\ddbzs$-derivatives
$$g_N:= \ddbz f_N  = \sum_{n=1}^N \lambda_n\mu_n $$
are bounded in $T_{s-1}=T_{\beta-2}$, i.e. that for some $K>0$ we have
$$  \sum_{n=1}^N \frac{\lambda_n}{\pi} \cdot
\int\frac{t\,d\mu_n(\zeta)}{
(t^2+ |z-\zeta|^2)^{\frac32}}
\le K t^{\beta-2},$$
whenever $z\in\C$ and $t>0$.

When $t\ge1$, we have the trivial estimate (independent
of $z$)
$$\lambda_n(P_t*\mu_n)(z,t)\le\frac{\lambda_n}{\pi}
\frac{M^{\beta}_*(E_n)}{t^2} \le \frac1{\pi 2^nt^2}, $$
so this gives $|P_t*g_N| \le K t^{-2}\le K t^{\beta-2}$.  

So to finish, fix $t\in(0,1)$, and choose
$m\in\N$ such that $2^{-m-1}\le t\le2^{-m}$,
take the $n$-th term in the sum,
and consider separately the ranges of $n$:
\\
case $1^\circ$: $n>m-2$, and case 
$2^\circ$: $n<m-2$,
\\ and the possible positions of $z$ in relation to
$A_n$.  

\smallskip\noindent Case $1^\circ$:
\\The trivial estimate also gives
$$   
\lambda_n\cdot(P_t*\mu_n)(z,t)
\le
\frac{M^\beta(A_n)}{\pi t^2}
\le 
\frac{(2^{-n})^\beta}{\pi t^2},
$$
so we get an estimate for the total
contribution from all the Case $1^\circ$ terms:
$$ t^{2-\beta}\sum_{n=m-2}^\infty 
\lambda_n\cdot(P_t*\mu_n)(z,t)
\le
\frac1{\pi}
\sum_{n=m-2}^\infty
2^{(m+1-n)\beta} = \frac{8^\beta}{\pi(1-2^{-\beta})},
$$

\smallskip\noindent
Case $2^\circ$:
\\ To deal with this we have to consider the position
of $z$ in relation to $A_n$.

There are at most three $n$ such that the distance
from $z$ to $A_n$ is less than $2^{-n-1}$. For these
we can use the uniform estimate 
$$t^{2-\beta}\cdot (P_t*\mu_n)(z,t) \le K, $$
which follows from the fact that $\mu_n$ has growth $\beta$.
(just write the value of $P_t*\mu_n(w)$ as a sum
of the integrals over the annuli
$$ \{z\in\C: 2^{m}t<|z-w|\le2^{m+1}t \}$$
from $0$ to $-\log_2t$ plus the integral
over the disc $\B(w,t)$).

For the remaining $n\in\{1,\ldots,m-3\}$, the estimate
$$ (P_t*\mu_n)(z,t) \le \frac{t\cdot M^\beta(A_n)}{\pi\cdot \dist(z,A_n)^3} $$
gives 
$$t^{2-\beta}\cdot (P_t*\mu_n)(z,t) \le (2^{\beta-3})^{m+1-n}, $$
so 
$$ t^{2-\beta}\sum_{n=1}^{m-3} 
\lambda_n\cdot(P_t*\mu_n)(z,t)
\le
3K+\sum_{n=1}^{m-3} 
(2^{\beta-3})^{m+1-n} = K, $$
another constant (depending on $\beta$), 
and we are done.

\medskip
Now consider the converse. Suppose
$\sum_n 2^n M^\beta_*(A_n\setminus U)<+\infty$.  We
want to show that $A^s(U)$ admits a continuous point evaluation at $0$.

If $V$ is an open subset of $U$, then
$A^s(U)$ is a subset of $A^s(V)$, so it 
suffices to prove the result for $U$
that are contained in $\mathring{\B}(0,\half)$.
We assume this is the case. 

We may choose radial functions $\psi_n\in\Test$
such that $\psi_n=1$ on $A_n$, $\psi_n=0$
off $A_{n-1}\cup A_n\cup A_{n+1}$, and
for each $k$ the sequence $(N_k(\psi_n))_n$
is bounded.  Let 
$$ \phi_n:= \frac{\psi_n}{\sum_{m=1}^\infty \psi_m}$$
on the complement of $\{0\}$, and $\phi_n(0)=0$. 
Then each $\phi_n\in\Test$, is zero 
off $A_{n-1}\cup A_n\cup A_{n+1}$,
the sequences $(N_k(\phi_n))_n$ are all bounded,
and $\sum_n\phi_n=1$ on the union of all
the $A_n$.

Fix a test function $\chi$ that equals $1$ on
$\B(0,\half)$ and is supported on $\B(0,1)$.

Fix $f\in A^s_0(U)$. 
We want to prove that $|f(0)|\le K\|f\|_s$, where $K>0$
does not depend on $f$.

We have $f(0)=(\chi\cdot f)(0)$, $\chi\cdot f\in A^s(U)$, and 
$\|\chi\cdot f\|_s\le K\|f\|_s$,  so it suffices to prove 
the estimate for $f\in A^s_0(U)$ having support in $\B(0,1)$.

Choose $N\in\N$ such that $f(z)$ is holomorphic
for $|z|<2^{2-N}$.  Define $\phi_0(z)$ to be
$1-\phi_N(z)$ when $|z|<2^{-N-1}$
and $0$ otherwise. Then $\phi_0\in\Test$,
$N_k(\phi_0)=N_k(\phi_N)$, and
 the test function
 $$ \phi:= \phi_0 + \sum_{n=1}^N \phi_n $$
is equal to $1$ near $\B(0,1)$. 

We have 
$$ f = \phi\cdot f = \Cau\left(\ddbz(\phi\cdot f)\right). $$
Since $\ddbz\phi=0$ on the support of $f$,
this equals
$$ \Cau\left(\phi\dd{f}{\bar z}\right)
= \Cau\left(\phi_0\dd{f}{\bar z}\right) + 
\sum_{n=1}^N \Cau\left(\phi_n\dd{f}{\bar z}\right)
= \sum_{n=1}^N \Cau(\phi_n\dd{f}{\bar z}),
$$ 
since $f$ is holomorphic on $\spt(\phi_0)$.

Take a test function $\psi$ 
that equals $1/z$
for $2^{-N}<|z|<2$. 

Applying Lemma \ref{L:B}, we have
$$ 
\Cau\left(\phi_n\cdot\dd{f}{\bar z}\right)(0) 
= -\left\langle \frac{\psi}{\pi}, \phi_n\cdot\dd{f}{\bar z}\right\rangle
= -\left\langle \frac{\psi\cdot\phi_n}{\pi}, \dd{f}{\bar z}\right\rangle.
$$

Thus
$$ f(0) = 
-\sum_{n=1}^N
 \left\langle \frac{\phi_n}{\pi z}, \dd{f}{\bar z}\right\rangle
$$
(Here, by $\phi_n/z$ we understand the test function
that equals $0$ at the origin and $\phi_n/z$
everywhere else in $\C$.)

Applying the Hausdorff content estimate from Lemma
\ref{L:M-beta-*}, we have
$$
\begin{array}{rcl}
 |f(0)| &\le& 
\dsty K\cdot \sum_{n=1}^N
N_3\left( 
\frac{\phi_n}{z}
\right)
\cdot
M^\beta_*\left(\spt(\phi_n\cdot \dd{f}{\bar z})\right) 
\cdot\left\|\dd{f}{\bar z}\right\|_{s-1} 
\\
&\le&
\dsty 
K\cdot \sum_{n=1}^{\infty}
2^n M^\beta_*\left((A_{n-1}\cup A_n\cup A_{n+1})\setminus U\right)
\cdot\left\|\dd{f}{\bar z}\right\|_{s-1}, 
\end{array}
$$ 
since 
$N_3(\phi_n/z) \le 2^{n+1}N_3(\phi_n)\le K2^n$.

Since $M^\beta_*$ is subadditive and 
$\dsty\left\|\dd{f}{\bar z}\right\|_{s-1}
\le K\|f\|_s\le K$, we get
$$ | f(0) | \le  
K\cdot \sum_{n=1}^{\infty}
2^n M^\beta_*(A_n\setminus U)\cdot \|f\|_s. 
$$

This completes the proof.

\begin{remark}
The proof actually shows that the
sum of the series is the dual norm of the point evaluation
$f\mapsto f(b)$, up to multiplicative constants
that depend only on $\beta$.
\end{remark}

\subsection{Proof of Theorems \ref{T:2},
\ref{T:3} and \ref{T:4}}
To prove Theorem \ref{T:2}, one can use exactly the same
argument, just replacing $M^\beta_*$ by $M^\beta$,
and using Lemma \ref{L:M-beta}
instead of Lemma \ref{L:M-beta-*}.

To prove the other two theorems, one just uses
the corresponding Cauchy-Pompeiu formulas for derivatives:
$$
\Cau(\mu_n)^{(k)}(0) = 
\frac{k!}{\pi}\int \frac{d\mu_n(z)}{z^{k+1}}
$$
for the \lq only if' direction, and
$$ f^{(k)}(0) = 
-\sum_{n=1}^N
 \left\langle \frac{k!\phi_n}{\pi z^{k+1}}, \dd{f}{\bar z}\right\rangle
$$
for the \lq if' direction.

\subsection{Proof of Theorem \ref{T:5}}
The point is that
a distribution $f\in C_s$ satisfies
$\Delta f=0$ on the open set $U$
if and only if $\dd{f}{z}$ is holomorphic
on $U$, and the operator $\dd{}{z}$
maps $T_s$ into $T_{s-1}$,
and is inverted on $(T_{s-1})_\cs$ by the
\lq anti-Cauchy' transform. So the results
are just reformulations of Theorem \ref{T:1}.

\begin{remark}
This is an example of 
$1$-reduction, and one could also formulate
equivalent results about other elliptic operators.
In particular, $M^\beta_*$ is also the
capacity for $T_{\beta}$ and the operator
$\left(\ddbzs\right)^2$, which is associated to complex
elastic potentials, and it is the capacity
for $T_{\beta+2}$ (a space of functions that
are twice differentiable, but may have discontinuities
in the third derivative) and the operator
$\Delta^2$, associated to elastic plates.
\end{remark}

\section{Concluding remarks}

\subsection{}
In \cite[p.311]{Carleson-1950}
Carleson proved that for $0<\beta<1$
the $M^\beta$-null sets are the removable
singularities for the class
of $\Lip\beta$ \lq\lq multiple-valued holomorphic
functions having single-valued real part."
The expression in quotation marks is really code for
\lq\lq harmonic functions", so this is really the
first version of the fact that $M^\beta$ is the 
$\Delta$-$\Lip\beta$-$\cp$.  

In the same paper, Carleson proved a precursor to
Dolzhenko's theorem \cite{Dolzhenko} about 
removable singularities for $\Lip\beta$ holomorphic 
functions.  He left a little gap, between the Hausdorff
content and the nearby Riesz capacity, and Dolzhenko
closed the gap.

\subsection{Conjecture}
Recently \cite{OF-Sbornik} the author showed
that the existence of a continuous point
derivation on $A^s(U)$ at $b$, for some positive 
$s<1$, implies that the value of the derivation
may be calculated by taking limits of difference
quotients from 
a subset $E\subset U$ having full area
density at $b$.  In case $U$ also satisfies
an interior cone condition at $b$, 
the value may be calculated by taking
limits along the midline of the cone.  
It seems reasonable to hope that
for negative $s$, if $A^s(U)$ admits a continuous
point evaluation at $b$, then the value can
be calculated in a similar way, as
$$ \lim_{z\to b, z\in E} f(z) $$
for some $E\subset U$ having full area density at
$b$, and for segments $E\subset U$
(if any) along which nontangential approach to $b$
is possible.  In the case of $L^p$ spaces, results
along these lines have also been obtained
by Wolf \cite{Wolf1979} and Deterding \cite{Deterding1,
Deterding2, DeterdingThesis}.
See also \cite{Fernstrom1975, Wolf1978, McCarthyYang}.

\subsection{Question}
Suppose $F$ is an SCBS on $\R^d$ having the
strong module property
$$ \|\phi\cdot f\|_F \le K\cdot N_k(\phi)\cdot\|f\|_F,
\quad \forall \phi\in\mathcal{D}\ \forall f\in F,$$
for some positive constant $K$ and some nonnegative
integer $k$.
Define an inner capacity $c_{F,k}$ by 
the rule that for each compact $E\subset\R^d$ the value
$c_{F,k}(E)$ is the least nonnegative number $c$ such that
$$ |\langle\phi,f\rangle| \le 
N_k(\phi)\cdot\|f\|_F\cdot c
$$ 
whenever $\phi\in\mathcal{D}$, $f\in F$, and
$\spt(\phi\cdot f)\subset E$.
For example, if $F=L^\infty$, it is easy to see that
$c_{F,0}(E)$ is the $d$-dimensional Lebesgue measure of $E$,
whereas for $F=L^1$, $c_{F,0}(E)=1$ for all $E$.

The question is this:
\emph{For which $F$ and $k$ is it the case that 
\\$c_{F,k}\le K\cdot(1$-$F$-$\cp)$
for some constant $K$?} 

Recall that for compact $E\subset\R^d$,
$$ \textup{$1$-$F$-$\cp$}(E) := \inf\{|\langle\chi,f\rangle|:
\|f\|_F\le 1, \spt(f)\subset E \}, $$
where $\chi\in\mathcal{D}$ is any fixed test function
such that $\chi=1$ on $E$.

We have seen that this holds for $F=T_s$, $s\in\R$.  Does
it hold for all SCBS having the strong module property?

\subsection{Question}
If an SCBS $F$ has the order $k$ strong module property, when is there
an SCBS locally-equal to $DF$ that has the strong
order $k+1$ module property?  And what about
$\sint F$?

$DF_\loc$ is the Frechet space topologised by the 
seminorms defined by
$$ \|f\|_n = \inf\left\{
\|g_1+\cdots+g_d\|_{F(\B(0,n))}: 
g_j\in F, f=\dd{g_1}{x_1}+\cdots+\dd{g_d}{x_d}
\right\}.
$$

Note that if $d=2$, and $F$ is weakly-locally invariant under
Calderon-Zygmund operators (or just under the Beurling transform),
then $DF_\loc$ is topologised by the
seminorms
$$ \|f\|_n = \inf\left\{
\|g\|_{F(\B(0,n))}: 
g\in F_\cs, f=\dd{g}{\bar z}
\right\},
$$
and this implies that each $\|\phi\cdot f\|_n$
is dominated by $N_{k+1}(\phi)\cdot \|f\|_n$,
because
$$ \phi\cdot \dd{g}{\bar z} = \ddbz{(\phi\cdot g)}
- \left(\dd{\phi}{\bar z}\right)\cdot g.$$
This property is a kind of local version of
the strong module property.

\subsection{Acknowledgment} The author is grateful to the
referee for a careful reading of the typescript and
for corrections and suggestions that  materially improved the 
exposition.

\bibliographystyle{amsplain}

\end{document}